\documentclass[reqno]{amsart}
\usepackage{amsfonts}
\usepackage{amsmath,amsxtra,amssymb,latexsym,amscd ,amsthm}
\usepackage{ucs}
\usepackage[mathscr]{eucal}
\usepackage{graphicx}
\usepackage{color}
\usepackage{graphics}
\usepackage{xypic,tikz}
\usepackage[all]{xy}
\UseComputerModernTips
\usepackage{hyperref}
\usepackage{setspace}
\usepackage{tikz-cd}

\usepackage{fancyhdr}

\usepackage{setspace}
\numberwithin{equation}{section}

\newtheorem{thm}{Theorem}[section]
\newtheorem{lem}[thm]{Lemma}

\newtheorem{pro}[thm]{Proposition}
\newtheorem{dfn}[thm]{Definition}

\theoremstyle{remark}
\newtheorem{rem}[thm]{Remark}

\title{Whitney stratifications and the continuity of local Lipschitz Killing curvatures}

\author{Nhan Nguyen and Guillaume Valette}

\address{Nhan Nguyen, Centre de Math\'ematiques et Informatique, Aix-Marseille Universit\'e, 39 rue Joliot-Curie, 13453 Marseille Cedex 13, France\
}
\email{nguyenxuanvietnhan@gmail.com}

\address{Guillaume Valette, Instytut Matemaczny PAN, Division in Krakow, ul. Sw. Tomasza, 30, 31-027 Krakow, Poland
}
\email{gvalette@impan.pl}

\newcommand{\bb}{\mathbb}
\newcommand{\al}{\mathcal}
\newcommand{\bff}{\mathbf}

\newcommand{\loc}{{\rm loc}}

\newcommand{\cl}{{\rm cl}}

\begin{document}
\maketitle

\parskip .12cm

\begin{abstract} We prove that local Lipschitz Killing curvatures of definable sets in a polynomially bounded o-minimal structure are continuous along strata of Whitney stratifications and locally Lipschitz if the stratifications are (w)-regular. 
\end{abstract}

\section{Introduction}
The density of a set $A\subset \bb R^n$ at a point $x \in \bb R^n$ is the following limit if it exists
$$\Theta_d(A, x) = \lim_{t \to 0}\frac{\al H^d(A \cap \bff B_{(x, r)}^n)}{\mu_d r^d},$$
where $\bff B^n_{(x,r)}$ is the ball in $\bb R^n$ centered at $x$ of radius $r$, $\al H^d$ is the $d$-dimensional Hausdorff measure and $\mu_d$ is the volume of the unit ball of dimension $d$. 

The existence of density of subanalytic sets was proved by  K. Kurdyka and G. Raby in \cite{kr}. In fact, if $A$ is a smooth manifold then density of  $A$ is $1$ at every point in $A$.  The notion of density  of $A$, therefore, not interesting at regular points.  If $A$ is a stratified set meaning the set $A$ together with its stratification, all singular points of $A$ will be in strata of dimensions less than $d$. It is natural to ask how the density  of $A$ varies along  those strata. 
In 1988, D. Trotman conjectured that the density of subanalytic sets is continuous on the strata of Whitney stratifications (stratifications satisfying Whitney's (b)-regular condition). G. Comte, in his thesis (1998), proved this conjecture for  (w)-regular stratifications (see also \cite{comte1}).  In \cite{valette1}, G. Valette gave  the positive answer for the conjecture. He also proved that the density is locally Lipschitz if the stratification is (w)-regular.

G. Comte and M. Merle \cite{cm} generalized Comte's result to what are called local Lipschitz Killing curvatures, introduced by A. Bernig and L. Br\"ocker \cite{bb}. To be precise,  let $A$ be a compact subset of $\bb R^n$, for $k \in \{0, \ldots, n\}$,  the $k$-th Lipschitz Killing of the set $A$ is defined as follows
\begin{equation}\label{fm_lips_curv_1} \Lambda_k(A): = c(n,k)\int_{P \in \bb G_n^k} \int_{x\in P} \chi (A \cap \pi_P^{-1}(x)) d\al H^k(x) dP, 
\end{equation}
where $\chi$ is  the Euler-Poincar\'e characteristic,  $dP$ is the standard probability measure of the Grassmannian $\bb G_n^k$ ,  $\pi_P$ is the orthogonal projection from $\bb R^n$  onto $P$ and  $c(n,k) = \Gamma(\frac{n+1}{2})\Gamma(\frac{1}{2})/\Gamma(\frac{n-k+1}{2})\Gamma(\frac{k+1}{2})$ with $\Gamma(s) = \int_0^\infty e^{-t}t^{s-1}dt$ .  Let $x\in A$. If the limit 
\begin{equation}\label{fm_local_Lip_cuv}
\Lambda_k^{\loc}(A, x): = \lim_{r\to 0}\frac{1}{\mu_k r^k} \Lambda_k(A \cap \bff B^n_{(x,r)})
\end{equation}
exists we call it  the $k$-th \textbf{local Lipschitz Killing curvature} of $A$ at $x$.  We would like to notice,  as a consequence of the Cauchy-Crofton formula, that $\Lambda_d(A, x) = \Theta_d(A, x)$.  In \cite{cm}, the authors proved that local Lipschitz Killing curvatures  of subanalytics sets exist and are continuous along strata of (w)-regular stratifications. 

In this paper we deal with the question whether the result of Comte and Merle still holds for Whitney stratifications.  We consider the problem in the framework of o-minimal structures which is considered as a generalization of semialgebraic and subanalytic geometries. By improving the techniques developed in \cite{valette1} to study the invariance of the density, we prove that local Lipschitz Killing curvatures of a definable sets in a polynomially bounded o-minimal structure are continuous along strata of Whitney stratifications. This result does not hold for o-minimal structure which is not polynomially bounded. An example in \cite{tv} show that, in general, the density of a definable set is not continuous along strata of Whitney stratifications.  We show furthermore that if the stratifications are (w)-regular, then the local Lipschitz Killing curvature are locally Lipschitz. This fact is true for every o-minimal structure, it is not necessary to assume to be polynomially bounded.

 In Section 2, we recall definitions of o-minimal structures and stratification of definable sets. 
 Section 3 presents some preliminary results of definable sets that is a preparation for proving Proposition \ref{prop_main_Killing} which is the key of the paper. The results of the continuity of local Lipschitz Killing curvatures for Whitney stratifications and for $(w)$-regular stratifications are presented in Section 4. 

\section{Definable stratifications}
\subsection{O-minimal structures}
A \textbf{structure} on the real closed  field $(\bb R,+,.)$ is a family $\al D = (\al D_n)_{n\in \bb N}$ satisfying the following properties:
\begin{itemize}
\item[(1)] $\al D_n$ is a boolean algebra of subsets of $\bb R^n$,

\item[(2)] If $A \in \al D_n$ then $\bb R \times A \in \al D_{n+1}$ and $A \times \bb R \in \al D_{n+1}$,

\item[(3)] $\al D_n$ contains the zero sets of all polynomials in $n$ variables,

\item[(4)] If $A \in \al D_n$ then its image under projection onto the first $n-1$ coordinates in $\bb R^{n-1}$ is in $\al D_{n-1}$.
\end{itemize}
A structure $\al D$ is said to be\textbf{ o-minimal} if in addition
\begin{itemize}
\item[(5)] Any set $A \in \al D_1$ is a finite union of open intervals and points.
\end{itemize}

Elements of $\al D_n$ for any $n$ are called \textbf{$\al D$-sets} (or definable sets) of $\al D$. A map between two $\al D$ sets is said to be a \textbf{$\al D$-map} (or definable map) if its graph is a $\al D$-set. 

A structure $\al D$ is said to be \textbf{polynomially bounded} if for every $\al D$-function $f: \bb R \to \bb R$, there exist $a > 0$ and $n \in \bb N$ such that $|f(x)| \leq x^n$ for all $x > a$. 

The class of semialgebraic sets, the class of globally subanalytic sets are examples of polynomially bounded o-minimal structures. We refer the reader to \cite{dries, dm1, coste, loi3} as good references for studying more about o-minimal structures. In the paper, we will use  the following properties of $\al D$-sets without citations.
 
1) Uniform Bounded on Fibres (see \cite{dries}, Chapter 3, (2.13))

2) Curve Selection (see \cite{dries}, Chapter 6, (1.5))

3) Hardt's triviality theorem (see \cite{dries}, Chapter 9, (1.7) or \cite{coste}, Theorem 5.22).

\subsection{Stratifications}
\begin{dfn}\rm
Let $A$ be a $\al D$-subset of  $\bb R^n$. Let $p \in \bb N$. A \textbf{$C^p$ $\al D$-stratification} (or a stratification for simplicity)  of $A$ is a partition $\al S = \{S_\alpha\}_\alpha$ of $A$ into finitely many $C^p$, connected $\al D$-submanifolds of $\bb R^n$, called \textbf{strata}, such that the frontier condition is satisfied meaning  if $S_\alpha, S_\beta \in \al S$, $S_\alpha\neq S_\beta$, $\cl(S_\alpha) \cap S_\beta \neq \emptyset$, then $S_\beta \subset S_\alpha$.
\end{dfn}
If we denote by $S^k$ the union of strata of dimensions $\leq k$ of the stratification $\al S$,  then $A$ can be described as filtration of skeletons 
$$A = S^d \supseteq S^{d-1} \supseteq \ldots \supseteq S^l \neq \emptyset$$
such that each difference $\mathring{S}^k = S^k \setminus S^{k-1}$ is an $k$-dimensional $C^p$ $\al D$-submanifold of $\bb R^n$ or empty. Connected components of $\mathring{S}^k$ coincide with strata of dimension $k$ of $\al S$. 

The set $A$ together with its stratification $\al S$ is called \textbf{stratified set} and denoted by $(A, \al S)$. A vector field $v$ defined on $A$ is called \textbf{stratified vector field} with respect to the stratification $\al S$ if $v(x) \in T_x S$, $x \in S \in \al S$. 

Suppose that ($\gamma$) is some regularity condition defined on pairs of submanifolds of $\bb R^n$. A stratification of $A$ is said to  be  ($\gamma$)-regular if the condition $(\gamma)$ is satisfied for every pair of strata of the stratification. In the following, we recall definitions of regularity conditions which we will deal with in next sections. 

Let $X, Y$ be $C^1$  $\al D$-submanifolds of $\bb R^n$. Let $z \in \cl(X) \cap Y$. 

\textbf{Whitney condition (b)}--- for any sequence $\{x_k\}_{k \in \bb N}$ in $X$ and any sequence $\{y_k\}_{k \in \bb N}$ in $Y$, each converging to $z$ such that the sequence of tangent spaces $\{T_{x_k}X\}_{k \in \bb N}$ converges to $\tau \in \bb G_n^{\dim X}$, and the sequence of vectors $\frac{x_k - y_k}{\|x_k - y_k\|}$ converges to a unit vector $v$, one has $v \in \tau$. 

\textbf{Kuo's ratio test (r)}--- $\delta (T_{\pi_Y(x)}Y, T_xX)\ll \dfrac{\|x-\pi_Y(x)\|}{\|x - z\|} ,$ where $x\in X$ and $x$ is converging to $z$.

\textbf{Verdier condition (w)}--- there exist a neighborhood $U_z$ of $z$ in $\bb R^n$ and a constant $C>0$ such that 
\begin{equation*}
\delta (T_yY, T_xX) \leq C \|x - y\|, \hspace{1cm} \forall x \in U_z \cap X, \forall y \in U_z \cap Y.
\end{equation*}
Here, $\pi_Y$ denotes the locally orthogonal projection onto $Y$ and \[ \delta (M, N):= 
\begin{cases}
 \sup_{\substack{\\x \in M, \|x\| = 1}} \|x - P_N(x)\|, & \text{if } M \not\equiv 0 \\
  0, & \text{if } M \equiv 0\\
 \end{cases}
 \]
for $M, N$ are vector subspaces of $\bb R^n$, where $P_N$ is the orthogonal projection from $\bb R^n$ onto $N$. 
 
\begin{rem} \label{rem_stra}
 In o-minimal setting, we have $\rm (w) \Rightarrow (r) \Rightarrow (b)$. Moreover, if $\al D$ is polynomially bounded and  $\dim Y = 1$ then $\rm (r) \Leftrightarrow  (b)$ (see \cite{tv}, \cite{nhan1}).
\end{rem}

\section{Preliminary results of $\al D$-sets}
Let $A$ be a $\al D$-subset of $\bb R^n$ (consider $A$ as a family of $\al D$-subsets of $\bb R^{n-k}$ parameterized by $\bb R^k$). Let $U \subset \bb R^k$.  We denote by $A|_U := \{x = (q,t)\in \bb R^{n-k}\times \bb R^k : x \in A, t \in U\}$. Let $t \in \bb R^k$. The set $A_t: = \{q \in \bb R^{n-k}: (q,t) \in A\}$ is called the\textbf{ fibre of $A$ at the point $t$}. Let $\varepsilon> 0$. The \textbf{neighborhood of $A$ of radius $\varepsilon$ }is defined as follows
\begin{equation}
\al N(A, \varepsilon):= \{x \in \bb R^n: d(x, A) \leq \varepsilon\},
\end{equation}
where $d$ denotes the Euclidean distance in $\bb R^n$.  Assume that $\dim A = l$. For $r> 0$, we define
\begin{equation}
\psi(A, r): = \al H^l \left(A \cap \bff B^n_{(0,r)}\right).
\end{equation}
\begin{pro}\label{pro_vol}
Let $A\subset \bb R^n \times \bb R^m$ be a $\al D$-set. Consider $A$ as a family of $\al D$-sets in of dimensions at most $l$ in $\bb R^n$ parameterized by $\bb R^m$.  Then, there exists a constant $C> 0$ such that for any $r > 0$ and any $t \in \bb R^m$, we have
\begin{enumerate}
\item $\psi(A_t, r) \leq C r^l$.
\item  If $l < n$, $\forall \varepsilon > 0$, then 
$$\psi(\al N( A_t, \varepsilon), r) \leq C r^{n-1} \varepsilon.$$
\end{enumerate}
\end{pro}
 
\begin{proof} We follow closely the proof of Propositions 3.06 and 3.07 in \cite{valette1}. 

(i)--- In the case $l = n$, the set $A_t\cap \bff B^n_{(0, r)}$ being included in the ball $\bff B^n_{(0,r)}$ for all $t\in \bb R^m$, the result is obvious (the constant $C$ is $\al H^l(\bff B^l_{(0,1)}$). If $l < n$, by removing a $\al D$-subset of dimension less than $l$, we can consider $A_t$ as a finitely disjoint union of graphs of Lipschitz mappings after a possible change of coordinates (the number of these graphs are bounded by a constant independent of $t$, see \cite{kp}, Proposition 1.4). The volume of such a graph is equivalent to the volume of its image under the projection onto $\bb R^{l}$. The conclusion then follows from the case $l = n$. 

(ii)--- We denote by
$$A_t(\alpha): = \{ x\in \bb R^n: d(x, A_t) = \alpha\}.$$
Since $\dim A_t < n$, $A_t(\alpha)$ is a $\al D$-set of dimension $n-1$. By the case (i), $\psi(A_t(\alpha), r) \leq C r^{n-1}$. Then, 
\begin{align*}
\psi(\al N( A_t, \varepsilon), r) &= \int_{\al N( A_t, \varepsilon)\cap \bff B^n_{(0,r)}} d\al H^{n} \\
& \leq \int_0^{\varepsilon} \psi(A_t(\alpha), r) d \al H^1(\alpha) \\
& \leq C r^{n-1} \varepsilon. 
\end{align*}

\end{proof}


\begin{lem}\label{lem_retract}
Let $A$ be a closed $\al D$-subset of $\bb R^n$ containing the origin. Suppose $\Sigma$ is a Whitney stratification of $A$ and $\{0\} \in \Sigma$. Then, there exists an $r_0 > 0$ such that for every $0< r< r' \leq  r_0$, there is a deformation retract from $A \cap \bff B^n_{(0, r')}$ onto $A \cap \bff B^n_{(0, r)}$ which preserves strata of the stratification $\Sigma$, i.e. there is a continuous mapping
$$F: A \cap \bff B^n_{(0, r')} \times \left[0, 1 \right] \to A \cap \bff B^n_{(0, r')}$$
such that $F(x, 0) = x, F(x, 1) \in A \cap \bff B^n_{(0, r)}$, $F(x,1)|_{x \in {A \cap \bff B^n_{(0, r)}}}= x$ and $F|_{S \times[0,1]} \subset S,\forall S \in \Sigma$.  

Moreover, \begin{equation}\label{fm_retract}
\|F(x, t) - x\| \leq 2t| r' - r|.
\end{equation} 
\end{lem}

\begin{proof}

Let $\rho: \bb R^n \to \bb R$, $x \mapsto \|x\|$ the distance function to the origin. 
Choose $r_0$ such that for every $0 < r \leq  r_0$, $\bff S^{n - 1}_{(0, r)} \pitchfork S$, $\forall S \in \Sigma$. The collection $\Sigma': = \{\bff B^n_{(0,r_0)}\cap S,  S \in \Sigma\}$ is then a Whitney stratification of $A \cap \bff B^n_{(0,r_0)}$.  Moreover, the restriction of $\rho$ to each stratum of $\Sigma'$ is submersive. 

For $x\in S$, $S \in \Sigma'$, set $v(x):= P_x (\partial_x \rho)$ where $P_x : \bb R^n \to T_{x}S$ is the orthogonal projection from $\bb R^n$ onto the tangent space of $S$ at $x$. Since $\rho|_S$ is a submersion, $v(x) \neq 0$, $\forall x \neq 0$.

Given an $\varepsilon > 0$, because $\Sigma'$ is a Whitney stratification, shrinking $r_0$ if necessary, we can assume that $\forall x\in A \cap \bff B^n_{(0,r_0)}$, $\|v(x) - \partial_x \rho\| \leq \varepsilon$.

Consider $\Sigma'$ as a filtration of skeletons $A \cap \bff B^n_{(0,r_0)} = S^d\supset S^{d-1}\supset \ldots \supset \{0\}$, where $ d = \dim A \cap \bff B^n_{(0,r_0)}$.  Notice that, generally, $v$ is not continuous on $S^d$ even though its restriction to each stratum in $S^d$ is continuous. We first construct a continuous stratified vector field on $S^d$, say $w$, by induction on skeletons so that 
$$\|w(x) - \partial_x\rho\| \leq c\varepsilon$$
where $c$ stands for some constant.

For $d = 0$, we take $ w =\mu = 0$.  Otherwise, write $S^d = \mathring{S}^d \cup S^{d-1}$ where $\mathring{S}^d$ is the union of the strata of dimension $d$ in $S^d$. By the inductive hypothesis, $\mu$ is a continuous stratified vector field on $S^{d-1}$ and $\|\mu(x) -  \partial_x\rho\| \leq c\varepsilon$. By a result of du Plessis \cite{plessis}, $\mu$ can be extended to a continuous stratified vector field on $S^d$. We will use the same notation $\mu$ for this extension. Since $\mu$  and $\partial_x\rho$ both are continuous on $S^d$, for each point $y \in S^{d-1}$ we can choose a neighborhood $U_y$ in $\bb R^n$ such that for any $x \in S^d \cap U_y$, we have $\|\mu(x) - \mu(y)\| \leq \varepsilon$  and $\|\partial_x \rho - \partial_y \rho\| \leq \varepsilon$. The union $\bigcup_{y \in S^{d-1}} U_y$ is an open neighborhood of $S^{d-1}$ in $\bb R^n$.  Define $T: =\bigcup_{y \in S^{d-1}}\left( U_y \cap \mathring{S}^d \right)$, and call it an open neighborhood of $S^{d-1}$ in $\mathring{S}^d$. Let $T'$ be another open neighborhood of $S^{d-1}$ in $\mathring{S}^d$ such that the closure of $T'$ in $\mathring{S}^d$ is contained in $T$ . Then we can choose a smooth partition $\{g_1, g_2\}$ of unity whose support refines $\{ \mathring{S}^d \setminus T',T\}$, and define 
\[ 
w(x) = 
\begin{cases}
\mu, &  x \in S^{d-1}\\
g_1 v + g_2 \mu,& x \in \mathring{S}^{d}
\end{cases}
\] 
It is clear that $w$ is a continuous stratified vector field. Now we show that  $\|w(x) - \partial_x \rho\| \leq c \varepsilon$. It suffices to check that the formula holds for every $x\in T$ since otherwise $w(x) = \mu(x)$ if $x \in S^{d-1}$ and $w(x) = v(x)$ if $x \in \mathring{S}^d \setminus T$ and obviously the formula holds. 

Suppose $x \in T$. By the construction of $T$, there is $ y \in S^{d - 1}$ such that $\|\mu(x) - \mu(y)\| \leq \varepsilon$  and $\|\partial_x \rho - \partial_y \rho\| \leq \varepsilon$. Hence, 
\begin{align*}
\|w(x) - \partial_x \rho\|  &= \|g_1(v (x) - \partial_x \rho) + g_2(\mu(x) - \partial_x \rho)\| \leq \varepsilon +  \|\mu(x) -     \partial_x \rho\|\\
& \leq \varepsilon +  \|\mu(x) - \mu(y)\| + \|\mu(y) -  \partial_y \rho\| + \| \partial_y \rho - \partial_x \rho\|\\
& \leq \varepsilon+ \varepsilon + c \varepsilon + \varepsilon  = (3 + c) \varepsilon.
\end{align*}
Since $w$ and $\partial_x\rho$ are continuous and $\|w(x) - \partial_x \rho\| \leq c\varepsilon$, $\| \partial_x \rho \| = 1$,
\[
\xi(x) =
\begin{cases} 
0, & \text{if } x =0\\
\frac{-w(x)}{\langle \partial_x \rho, w(x)\rangle}, & \text{othewise }
\end{cases}
\]

is well-defined and continuous on $S^d \setminus \{0\}$. We can, moreover, choose $\varepsilon$ small enough so that $ \| \xi(x) \| < 2$. 
 
Write $\Phi (x, t)$ as the flow generated by the vector field $\xi$. Since $\rho_{*}(\xi(x)) = -1$, if $x \in X\cap \bff B^n_{(0,r)}$ then $\Phi(x,s) \in X\cap \bff B^n_{(0,r-s)}$ for all $r \leq r_0$. The map defined as follows 
\[
F(x,t) = 
\begin{cases} 
\Phi(x, t(\|x\|- r)), & \text{if } x \in X\cap \left(\bff B^n_{(0,r')} - \bff B^n_{(0,r)}\right)\\
x, & \text{if } x \in X \cap \bff B^n_{(0,r)}
\end{cases}
\]
is actually the desired deformation retract.
\end{proof}

\begin{rem} As in the proof, we are able to require $|\xi|< C$ with any constant $C > 1$ by taking $\varepsilon$ sufficiently small.  
\end{rem}


\section{Continuity of local Lipschitz Killing curvatures}
\subsection{Lipschitz Killing curvatures of $\al D$ sets}
Let $A$ be a compact $\al D$-subset of $\bb R^n$.  By Uniform Bound on Fibers  and Hardt's triviality theorem, there is an integer $N \geq 0$ such that for every $P \in \bb G_n^k$, $0\leq k \leq n$, there is a partition of $P$ into $\al D$-sets 
$$K_{k,j}^P(A): = \{ x \in P: \chi \left(\pi_P^{-1}(x) \cap A\right) = j\},$$
$j \in \{-N, -N+1, \ldots, N\}$. The formula (\ref{fm_lips_curv_1}), therefore, becomes
\begin{equation}\label{fm_lips_curv_2}
\Lambda_k(A)= c(n,k)\int_{P \in \bb{G}_n^k} \sum_{j = 1}^N j \al H^k\left(K_{k,j}^P(A)\right) dP.
\end{equation}

Let $x \in A$.  By the same argument as in the proof of  Theorem 1.3 \cite{cm}, we can show that the local Lipschitz Killing curvatures $\Lambda^{\loc}_k(A, x)$ of $A$ at $x$ (see the formula (\ref{fm_local_Lip_cuv})) are well defined. It easy to verify that $\Lambda^{\loc}_0(A, x) = 1$ and $\Lambda^{\loc}_k(A, x) = 0$ for every $k > \dim A$.  

\subsection{Whitney condition (b)}\label{sec_whitney_b}
In this section, we suppose that $\al D$ is a polynomially bounded o-minimal structure,  $A$ is a closed $\al D$-subset of $\bb R^n$ with a Whitney stratification $\Sigma$. Let $Y$ be a stratum of $\Sigma$. We may assume locally that $Y = \{0\}^{n-k} \times \bb R^k$ where $k = \dim Y$ and  $\Sigma= \{Y, X_1, \ldots, X_m\}$ with $Y \subset \cl(X_i) \setminus X_i$, $\forall i \in \{1,\ldots, m\}$. 

\begin{lem}\label{lem_appro} Suppose that $\gamma(-\varepsilon, \varepsilon)$ is a $C^1$ $\al D$-curve in $Y$, $\gamma(0) = 0$. Then, there exist $\nu > 0$ and a germ of homeomorphism
$$h:  A|_{\gamma([0, \nu ])}\to  A_0 \times \gamma([0, \nu ]), \hspace{1cm} h(q, t) = (h_t(q),0, t),$$ and constants $C_1, C_2$ such that 
\begin{align*}
\|h_t(q) - q\| & \leq C_1 \|t\|^{1-e} r \\
\|h_t^{-1}(q) - q\| & \leq C_2 \|t\|^{1-e} r,
\end{align*} 
 $\forall (0,t) \in \gamma([0,\nu])$, $\forall q \in \bff B^{n-k}_{(0,r)}$ for which the above mappings are well-defined with $r$ sufficiently close to zero.
\end{lem}

\begin{proof}
We may write $\Sigma = \{Y= X_0, X_1, \ldots , X_m\}$. Let $Y' : = \gamma(-\varepsilon, \varepsilon)$. Denote by $\Sigma': =\{Y', X_1|_{Y'}, \ldots, X_m|_{Y'}\}$. Then, $\Sigma'$ is a Whitney stratification of $A|_{Y'}$ (consider in a neighborhood of $Y'$ in $\bb R^n$).

Since $\dim Y' = 1$ and $\Sigma'$ is a Whitney stratification, $\Sigma'$ satisfies the condition (r) along $Y'$ (see Remark \ref{rem_stra}). As proven in \cite{ot}, there exist  $0\leq  e< 1$ and a constant $C> 0$ such that 
\begin{equation}\label{eqq_1}
\delta(T_{\pi(x)} Y', T_x(X_i|_{Y'}))\|\pi(x)\|^e \leq C \|x - \pi(x)\|,
\end{equation}
where $x\in X_i|_{Y'} \cap U$ with $U$ is a small neighborhood in $\bb R^n$ of the origin,  $\pi: \bb R^n \to \bb \{0\}^{n-k}\times \bb R^k$, $\pi(q, t) = (0,t)$. 

Let $\mu (x): = -\gamma'(x)/\|\gamma'(x)\|$. Set $w(x): = P_x\left(\mu(\pi(x))\right) \in T_x (X_i|_{Y'})$, where $P_x$ denotes the orthogonal projection onto the tangent space $T_x (X_i|_{Y'})$, $x \in X_i|_{Y'}\cap U$. Notice that $w(x) = \mu(x)$ for every $x \in Y'$. It follows from (\ref{eqq_1}) that
\begin{equation}\label{equ_2.1}
\|w(x) - w(\pi(x))\|\|\pi(x)\|^e  \leq C \|x - \pi(x)\|.
\end{equation}
The condition Whitney (a) ensures that $w$  and $\pi_{*}(w(x))$ are bounded below away from $0$, where $\pi_{*}$ denotes the tangent map of $\pi$. Thus, the vector field $$v(x): = \frac{w(x)}{\|\pi_{*}(w(x))\|},$$ is well defined, satisfying ($\ref{equ_2.1}$) and $\pi_{*}(v(x)) = v(\pi(x)) = \mu(\pi(x))$. 

Denote by $\Phi$ the flow generated by the vector field $v$. Let $x= (q, t) \in A|_{Y'} \cap U$, $(0,t) \in Y'$, and denote by  $\sigma(t)$ the length of the arc $\gamma$ from the origin to the point $(0,t)$.  By shrinking $U$ if necessary, we may assume that $\sigma(t) \sim \|t\|$ for all $(0,t) \in Y' \cap U$. 

For $0 <  s \leq \sigma(t)$, let $f(s):=\|\Phi_x(s) - \pi(\Phi_x(s))\|$. We have 
$$f'(s)  = \frac{\langle \Phi_x'(s) - \pi (\Phi_x'(s)), \Phi_x(s) - \pi(\Phi_x(s))\rangle}{\|\Phi_x(s) - \pi(\Phi_x(s))\|}\leq 
\|\Phi_x'(s) - \pi(\Phi'_x(s))\|.$$
Here
$$\|\Phi_x'(s) - \pi(\Phi'_x(s))\|  = \|v(\Phi_x(s)) - v(\pi(\Phi_x(s)))\|.$$
Associating with ($\ref{equ_2.1}$),
\begin{equation*}\label{equ_2.3} 
|f'/f| \leq C\|\pi(\Phi_x(s))\|^{-e}.
\end{equation*}
Since $\pi(\Phi_x(s)) = \Phi_{\pi(x)} (s)$ and $\|\Phi_{\pi(x)} (s)\| \sim s$, $|f'/f| \lesssim s^{-e}$. By integrating with respect to $s$ over $[0, s]$ (note that $f(0) = \|q\|$) we get
\begin{equation*}\label{equ_2.4}
\exp \left(\frac{-s^{(1-e)}}{1-e}\right) \|q\| \lesssim f(s) \lesssim \exp \left(\frac{s^{1-e}}{1-e}\right) \|q\|.
\end{equation*}
or, equivalently, 
$$f(s) \sim \|q\|.$$
Write
$$ \Phi_{x}(\sigma(t)) = (  \Phi^1_{x}(\sigma(t)),  \Phi^2_{x}(\sigma(t))) \in \bb R^{n-k}\times \bb R^k.$$
Then, 
\begin{align*}
\|\Phi^1_x(\sigma(t)) - q\|&= \bigg\|\int_0^{\sigma(t)} \frac{d}{ds}(\Phi_x(s) - \pi(\Phi_x(s))ds \bigg\| 
= \bigg\| \int_0^{\sigma(t)} (\Phi'_x(s) - \pi(\Phi_x'(s))ds \bigg\| \\
&=  \bigg\| \int_0^{\sigma(t)} (v(\Phi_x(s)) - v(\pi(\Phi_x(s)))ds \bigg \| \\
&\leq \int_0^{\sigma(t)} \frac{\|\Phi_x(s) - \pi(\Phi_x(s))\|}{\|\pi(\Phi_x(s))\|^e}  \lesssim \int_0^{\sigma(t)} f(s) s^{-e} ds\\ 
 & \lesssim \|q\|\int_0^{\sigma(t)} s^{-e}ds  \lesssim \|q\|\sigma(t)^{1-e} \lesssim \|q\|\|t\|^{1-e} .
\end{align*}
Hence, the desired homeomorphism is  given by $h_t(q) = \Phi^1_{(q,t)}(\sigma(t))$. 
\end{proof}

The following proposition is the key result for the proof of the main theorem. 
 
\begin{pro}\label{prop_main_Killing} 
Fix $0 \leq l \leq n - k$. There exists $C > 0$ such that for every $\varepsilon> 0$, there exists a neighborhood $U_{\varepsilon}$ of $0$ in $Y$ such that $\forall (0, t)\in U_\varepsilon$, $\exists r_t >0$, $\forall P \in \bb G_{n-k}^l$ and for every $0 < r \leq  r_t$, there is  a $\al D$-subset $ \Delta(P,\varepsilon, r, t)$ of $P$ with 
$$\al \psi\left( \left( \Delta(P,\varepsilon, r, t)\right), r\right) \leq C \varepsilon  r^l$$
such that for any $x \in \left( \bff B^{n-k}_{(0,r)}\cap P\right) \setminus  \Delta(P,\varepsilon, r, t)$,
\begin{equation}
\chi \left(\pi_P^{-1}(x) \cap  A_t \cap \bff B^{n-k}_{(0,r)}\right) = \chi \left(\pi_P^{-1}(x) \cap  A_0 \cap  \bff B^{n-k}_{(0,r)} \right),
\end{equation}
where $\pi_P$ is the orthogonal projection form $\bb R^{n-k}$ onto $P$. 

\end{pro}

\begin{proof}  
Since $\Sigma$ is a Whitney stratification, for each $t \in Y$ there is a  $r_t > 0$ such that  $\forall r \leq r_t$, the collection $\{S_t \cap \mathring{\bff B}^{n-k}_{(0, r)}, S_t \cap  \bff S^{n-k-1}_{(0, r)}\}_{S \in \Sigma}$ forms a Whitney stratification of $A_t \cap \bff B^{n-k}_{(0,r)}$, denoted by $\al S_t^r$. By Lemma \ref{lem_retract}, shrinking $r_t$ if necessary, we may assume that for $0< r < r'< r_t$, there is a deformation retract 
$$F_t^{r, r'}:  A_t \cap \bff B^{n-k}_{(0, r')} \times [0,1] \to A_t \cap \bff B^{n-k}_{(0, r')} $$
from $A_t \cap \bff B^{n-k}_{(0, r')}$ onto $A_t \cap \bff B^{n-k}_{(0, r)}$  preserving the strata of $\al S_t^{r_t}$ and 
$$\|F_t^{r,r'}(q,s) - F_t^{r, r'}(q,0)\| \leq 2s|r'-r|.$$

For $P \in \bb G_{n-k}^l$, consider the restriction of $\pi_P$ to $A_0 \cap \bff B^{n-k}_{(0, r + 2\varepsilon r)} \cup A_t \cap \bff B^{n-k}_{(0, r + 3\varepsilon r)} $. By Hardt's triviality theorem, there exists a partition of $P \cap \bff B^{n-k}_{(0,r+ 3\varepsilon r)}$ into finitely many $\al D$-sets such that  $A_0 \cap \bff B^{n-k}_{(0, r + 2\epsilon r)} \cup A_t \cap \bff B^{n-k}_{(0,r +  3\varepsilon r)} $ is definably trivial along elements of the partition (with respect to the projection map $\pi_P$) and  the trivialization over these elements is compatible with all strata of  $\al S^r_0, \al S^{r + 2\varepsilon r}_0, \al S^r_t, \al S^{r+\varepsilon r}_t, \al S^{r + 3\varepsilon r}_t$. 

Denote by $\Delta_{P,r, t}^1, \ldots, \Delta_{P,r, t}^\nu$ the elements of dimension $l$ of the partition and $\partial \Delta_{P,r, t}^1, \ldots, \partial \Delta_{P,r, t}^\nu$ their corresponding topological boundaries. Set
$$  \Delta(P,\varepsilon,r, t):= \bigcup_{i=1}^{\nu} \al N (\partial \Delta_{P,r, t }^i, 10 \varepsilon r). $$
Clearly, $ \Delta(P,\varepsilon,r, t)$ is a $\al D$-subset of $P$ of dimension less than $l$. By Proposition \ref{pro_vol}, we have $$\psi\left( \Delta(P,\varepsilon, r, t), r\right) \leq C \varepsilon r^l$$ for some $C> 0$ (note that $C$ is  independent of $(P, \varepsilon, r, t)$).

Let $x \in P$. We define 
$$ A_t^r(x, \lambda) := A_t \cap \bff B^{n-k}_{(0, r)}\cap \pi_P^{-1}(\bff B^{n-k}_{(x, \lambda)} \cap P).$$

\textit{\underline{Claim}}: For  $x \in \left(\bff B^{n-k}_{(0,r)}\cap P\right) \setminus  \Delta(P,\varepsilon,r, t)$, the homomorphisms of homology groups induced by the following inclusion maps
\begin{align*}
U_1:= A_0^{r}(x,\varepsilon r)\hookrightarrow{} U_2:=A_0^{r + 2\varepsilon r}  (x, 3\varepsilon r) & \hspace{1cm} \text{(I)}\\
W_1:=A_t^{r + \varepsilon r}(x,2\varepsilon r)  \hookrightarrow W_2:= A_t^{r + 3\varepsilon r} (x,4\varepsilon r) & \hspace{1cm}\text{(II)}\\
W_3:=A_t^{r}(x,\varepsilon r)  \hookrightarrow W_2:= A_t^{r + 3\varepsilon r} (x,4\varepsilon r) & \hspace{1cm}\text{(III)}
\end{align*}
are isomorphisms. 

We will give the proof of (I) (using the same arguments  we can get the proofs of (II) and (III)). 

Let $x\in  \bff B^{n-k}_{(0, r)} \cap P \setminus \Delta(P, \varepsilon, r, t)$. There exists $j \in \{1, \ldots, \nu\}$ such that $ x\in \Delta^j_{P, r, t}$. By the definition of $\Delta^j_{P, r, t}$, $\bff B^{n-k}_{(x, 10\varepsilon r)}\cap P \subset \Delta^j_{P, r, t}$.  Since the trivialization of $ A_0 \cap \bff B^{n-k}_{(0, r+ 2\varepsilon)}\cup A_t \cap \bff B^{n-k}_{(0, r + 3\varepsilon r)}$ over $\Delta^j_{P, r, t} $ is compatible with all strata of  $\al S_0^r$ and $\al S_0^{r + 2\varepsilon r}$,  $\forall \varrho \in \{r, r + 2\varepsilon r\}$ and $\forall \lambda, \lambda'$ such that $0< \lambda < \lambda' \leq 10 \varepsilon r$, there is a deformation retract, denoted by $\Psi^{\varrho}_0(x, \lambda, \lambda')$, from $A_0^{\varrho}(x, \lambda')$ onto $A_0^{\varrho}(x, \lambda)$ which preserves the strata of $\al S_0^\varrho$. 

Set
$$
U_2': = A_0^{r+ 2\varepsilon r}(x, 7\varepsilon r)$$
and
\begin{align*}
V_1&: = \{ F_0^{r, r + 2\varepsilon r}(q, s), q\in U_2, s \in [0, 1]\}\\
V_2&: = \{ \Psi^{r}_0(x, \varepsilon r,  7 \varepsilon r)(q, s), q \in F_0^{r, r + 2\varepsilon r}(U_2, 1),  s \in [0, 1]\}\\
V&: = V_1 \cup V_2. 
\end{align*}
It is obvious that $U_1$ is a retract of $V$ by the map defined as follows
\[G: V \times [0,1] \to V, \hspace{0.5cm} G(q, s) = 
\begin{cases} 
F_0^{r, r + 2\varepsilon r}(q, 2s), & q \in V_1, s\leq \frac{1}{2}\\
\Psi^{r}_0(x, \varepsilon r,  7 \varepsilon r)(q, 2s - 1), & q \in V_2, s > \frac{1}{2}.
\end{cases}
\]

Consider the following commutative diagram of homology groups induced by inclusion maps. 
\begin{center}
\begin{tikzcd}
& & H_*(V)  \arrow[leftarrow]{lld}[swap]{\alpha'_*} \arrow[leftarrow]{ld}[description]{\mu_*}
\arrow{ldd}[description]{\beta'_*}\\
H_*(U_1) \arrow{rd}[description]{\gamma_*}  \arrow{r}[description]{\alpha_*}  &  H_*(U_2) \arrow{d}[description]{\beta_*} & \\
& H_*(U_2').
\end{tikzcd}
\end{center}
Since $U_1$ and $U_2$ are retracts of $V$ and $U_2'$ respectively, the maps $\alpha'_*$ and $\beta_*$ are isomorphisms. Hence, homomorphisms in the diagram above are isomorphisms. This establishes claim (I). 

Now we are ready to prove the proposition.

 \textit{\underline{Case 1:}}  $\dim Y = 1$. First, we choose a neighborhood $U_\varepsilon$ of $0$ sufficiently small so that Lemma  \ref{lem_appro} holds, this means there are $0\leq e <1$, $c > 0$, for every $(0,t) \in U_\varepsilon$ there exist $r_t > 0$ and a homeomorphism $h_t: (A_t, 0) \to (A_0, 0 )$ such that
$$\|h_t(q) - q\| \leq c \|t\|^{1-e} r, \hspace{1cm} \forall q \in A_t \cap \bff B^{n-k}_{(0, r)},  r \leq r_t.$$
Shrinking $U_\varepsilon$ if necessary, we can assume that $c\|t\|^{1-e} < \varepsilon$, $\forall (0,t) \in U_\varepsilon$. This implies that
$$ U_1 \subset h_t(W_1) \subset U_2 \subset h_t(W_2).$$
Consider the following commutative diagram induced by inclusion maps

\begin{center}
\begin{tikzcd}
H_*(U_1) \arrow{r}{{\iota_1}_*}
\arrow{rd}{}[swap]{u_*}
&H_*( h_t(W_1)) \arrow{d}{\iota_*}
\arrow{rd}{w_*}[swap]{}\\
& H_*(U_2) \arrow{r}{{\iota_2}_*} & H_*(h_t(W_2))
\end{tikzcd} \end{center}

By the claim, we have that $u_*$ and $w_*$ are isomorphisms. Since the diagram commutes, it is easy to show that  ${\iota}_*, {\iota_1}_*, {\iota_2}_*$ are isomorphisms . Finally,
\begin{align*}
H_*\left(\pi_P^{-1}(x) \cap  A_t \cap \bff B^{n-k}_{(0,r)}\right) & = H_* ( U_1) = H_*(h_t(W_1))= H_*(W_2)\\
& = H_*(W_3) = H_* \left(\pi_P^{-1}(x) \cap  A_0 \cap \bff B^{n-k}_{(0,r)}\right).
\end{align*}

\textit{\underline{Case 2:}} $\dim Y > 1$. We define 
\begin{align*}
\Omega: = \bigg\{ (0,t) \in Y:  \forall  \varepsilon > 0,  \exists \sigma > 0, \forall r \in (0,  \sigma), \forall P \in \bb G_{n-k}^l, \\
\bigg[\forall x \in \bff B^{n-k}_{(0, r)}\cap P \setminus \Delta(P, \varepsilon, r, t)
 \\ \Rightarrow \big[ \chi \left(\pi_P^{-1}(x) \cap  A_t \cap \bff B^{n-k}_{(0,r)}\right) = \chi \left(\pi_P^{-1}(x) \cap  A_0 \cap \bff B^{n-k}_{(0,r)} \right) \big] \bigg]\bigg\}.
\end{align*}
Since $\chi(P, \varepsilon, x, t, r) : = \chi \left(\pi_P^{-1}(x) \cap  A_t \cap \bff   B^{n-k}_{(0,r)} \right)$ is a $\al D$-function, $\Omega$ is a $\al D$-set. It suffices to prove that the set $\Omega$ contains a neighborhood of $0$.  This fact actually follows directly from Curve Selection and Case 1.
\end{proof}


\begin{thm}[Main Theorem]\label{thm_killing_b}
The local Lipschitz Killing curvatures of $A$ are continuous along the strata of $\Sigma$.
\end{thm}
\begin{proof}  We will work with the following family of $\al  D$-sets. 
$$\al A := \{(x,u,t) \in \bb R^{n-k} \times \bb R^k \times \bb R^k: (x, u + t) \in A\}.$$
It is obvious that $\al A$  is the image of $A\times \bb R^k$ under the linear isomophism $$\varphi: A\times \bb R^k \to \al A, \hspace{1cm} (q, u, t) \mapsto (q, u- t,t).$$ 
The germ of $A$ at $(0,t)$ can be viewed as the germ of $\al A_t$ at $(0,0)$. Therefore, the continuity of $\Lambda_l^{\loc}(A, t)$ along $Y$ is equivalent to the continuity of  $\Lambda_l^{\loc}(\al A_t, 0)$  along $\{0\}^n \times \bb R^k$. 

Set
$$\Delta: = \{(u, t) \in \bb R^k \times \bb R^k: u = t\}.$$

Since $\Sigma = \{ Y = \{0\}^{n-k} \times \bb R^k, X_1, \ldots, X_m\}$ is a Whitney stratification of $A$ , the collection $ \{\{0\}^{n-k} \times \bb R^k\times \bb R^k, X_1\times \bb R^k, \ldots, X_m\times \bb R^k\} $ is a Whitney stratification of $A \times \bb R^k$,  and so is $\Sigma': =  \{\{0\}^{n-k}  \times \Delta, \{0\}^{n-k} \times (\bb R^k\times \bb R^k \setminus \Delta) , X_1\times \bb R^k, \ldots, X_m\times \bb R^k\}$. We denote by $\Sigma''$ the collection of images of all strata of $\Sigma'$ under the map $\varphi$. Then, $\Sigma''$ is a Whitney stratification of $\al A$ containing $\varphi(\{0\}^{n-k} \times \Delta) = \{0\}^n \times \bb R^k$  as a stratum.  It suffices to show that  $\Lambda_l^{\loc}(\al A_t, 0)$ is continuous  in $t$ along this stratum. 

Applying Proposition \ref{prop_main_Killing} to the stratified set $(\al A, \Sigma'')$, with $0\leq l \leq n$ fixed, there is $C > 0$, for every  $\varepsilon > 0$, there is  a neighborhood $U_\varepsilon$ in $\{0\}^n \times \bb R^k$ of the origin such that for any $(0,t) \in U_\varepsilon$, there is a $r_t> 0$ such that for any $0< r \leq  r_t$, for every $P \in \bb G_{n}^l$, there exists a  $\al D$-subset $\Delta(P,\varepsilon, t,r)$ of $P$ with
$$ \psi(\Delta(P,\varepsilon, t,r), r) \leq C \varepsilon r^l$$
such that for any $x \in \left(\bff B^n_{(0,r)} \cap P\right) \setminus  \Delta(P,\varepsilon, t,r)$:
\begin{equation*}
\chi (\pi_P^{-1}(x) \cap \al A_t \cap \bff B^n_{(0,r)}) = \chi (\pi_P^{-1}(x) \cap \al A_0 \cap \bff B^n_{(0,r)}).
\end{equation*}
It follows from (\ref{fm_lips_curv_2}) that for any $j$ and for any $P \in \bb G_{n }^l$, we have
$$ \psi\left(  K_{l,j}^P(\al A_0 \cap \bff B^n_{(0,r)}) \setminus K_{l,j}^P(\al A_t \cap \bff B^n_{(0,r)}), r\right) \leq C \varepsilon r^l$$ and,
$$ \psi\left( K_{l,j}^P(\al A_t \cap \bff B^n_{(0,r)}) \setminus K_{l,j}^P (\al A_0 \cap \bff B^n_{(0,r)}), r\right) \leq C \varepsilon r^l.$$
Thus, we get
$$\vert \psi\left( K_{l,j}^P(\al A_0 \cap \bff B^n_{(0,r)}, r\right) - \psi\left( K_{l,j}^P(\al A_t \cap \bff B^n_{(0,r)}), r\right)\vert \leq C \varepsilon r^l.$$
By formula (\ref{fm_local_Lip_cuv}), 
$$\vert \Lambda_l\left(\al A_0 \cap \bff B^n_{(0,r)}\right) - \Lambda_l\left(\al A_t \cap \bff B^n_{(0,r)}\right)\vert \leq C \varepsilon r^l.$$
Dividing by $r^l$, we obtain
$$\vert \Lambda_l^{\loc}(\al A_0, 0) - \Lambda_l^{\loc}(\al A_t, 0)\vert \leq C\varepsilon.$$ 
The theorem is proved.

\end{proof}

\subsection{Kuo-Verdier condition (w)}
In this section,  we assume that $\al D$ is an arbitrary o-minimal structure and $\Sigma$ is a (w)-regular stratification of $A$.  The other hypothesises remain as in Section \ref{sec_whitney_b}.  We first establish results of the same types as Lemma \ref{lem_appro} and Proposition \ref{prop_main_Killing}. 

\begin{lem}\label{lem_appro_w}
There exist a neighborhood $U$ of $0$ in $Y$ and constants $C_1, C_2 > 0$ such that for every $(0, t')$ in $U$, there is a germ of homeomorphism 
$$ h: A|_U \to A_{t'} \times U, \hspace{1cm} h_{t}(q, t)= (h_t(q), 0, t)$$
satisfying
\begin{align*}
\|h_{t}(q) - q\| &\leq C_1\|t - t'\|r  \\
\|h_{t}^{-1} (q) - q\| & \leq  C_2\|t - t'\| r
\end{align*}
where  $q \in \bff B^{n-k}_{(0, r)}$ for which the mapping $h_{t}$ is well-defined, $r$ is sufficiently small. 
\end{lem}

\begin{proof}The argument here is classical which can be found in \cite{mather} \cite{verdier}. Consider the coordinate vector fields $\partial_1, \ldots, \partial_k$ in $\{0\}^{n-k} \times \bb R^k$. There are corresponding rugose vector fields $\tilde{\partial_1}, \ldots, \tilde{\partial_k}$ on $A$ (consider in a neighbohood of $0$) such that 
$$\pi_*(\tilde{\partial_\alpha}) = \partial_\alpha, \hspace{1cm} \alpha \in \{1, \ldots, k\},$$
where $\pi: \bb R^{n-k} \times \bb R^k \to \{0\}^{n-k} \times\bb R^k$, $\pi(q, t) = (0,t)$ (see \cite{verdier}, Proposition 4.6).  

By the rugosity of $\tilde{\partial_\alpha}$, there is a neighborhood $V$ of $0$ such that for all $x \in V \cap A$,  
\begin{equation}\label{fm_w}
\| \tilde{\partial_\alpha} (x)-  \tilde{\partial_\alpha}(\pi(x) )\| \leq C \|x - \pi(x)\|.
\end{equation}
Denote by $\Phi_\alpha$ the flow generated by the vector field $\tilde{\partial_\alpha} $.  We write 
$$\Phi_\alpha(x,s) = (\Phi_\alpha^1(x, s) , \Phi_\alpha^2(x, s)) \in \bb R^{n-k} \times \bb R^k.$$
Since $\tilde{\partial_\alpha} $ satisfies the formula (\ref{fm_w}), by the same computation as in the proof of Lemma \ref{lem_appro} we have
$$\|\Phi_\alpha^1(x, s) - q\| \leq C |s|\|q\|, \hspace{1cm}  x = (q, t).$$
Define
$$h(x)  = (\Phi_{1}(\ldots (\Phi_{k-1}( \Phi_{k} (x, t_k' - t_k), t_{k-1}' - t_{k-1}),\ldots),  t_1' - t_1), t), $$
where $t = (t_1, \ldots, t_k)$ and $t' = (t_1', \ldots, t_k')$. 
It is easy to check that $h$ is the desired homeomorphism. 
\end{proof}

\begin{pro}\label{prop_main_killing_w} 
Fix $0 \leq l \leq n-k$. There exist a constant $C > 0$ and a neighborhood $U$ of $0$ in $Y$ such that for every $(0,t)$ and $(0, t')$  in $U$, $\exists r_{t, t'} >0$, for every $P \in \bb G_{n-k}^l$ and $0 < r \leq  r_{t, t'}$, there is  a $\al D$-subset $ \Delta(P,r, t, t')$ of $P$ with 
$$\al \psi\left( \left( \Delta(P,r, t, t')\right), r\right) \leq C \|t - t'\|  r^l$$
such that for any $x \in \left(\bff B^{n-k}_{(0,r)}\cap P\right) \setminus  \Delta(P, r, t, t')$,
\begin{equation*} 
\chi \left(\pi_P^{-1}(x) \cap  A_t \cap \bff B^{n-k}_{(0,r)}\right) = \chi \left(\pi_P^{-1}(x) \cap  A_{t'}\cap \bff B^{n-k}_{(0,r)} \right).
\end{equation*}
\end{pro}
\begin{proof}
Choose a neighborhood $U$ of $0$ in $Y$ sufficiently small so that Lemma  \ref{lem_appro_w} holds, i.e.,   there exists $c > 0$, for every $(0,t)$ and $(0, t')$ in $U$, there is a homeomorphism $h_{t, t'}: A_t \to A_{t'} $ and $r_{t, t'} > 0$ such that
$$\|h_{t, t'}(q) - q\| \leq c \|t - t'\| r, \hspace{1cm} \forall q \in A_t \cap \bff B^{n-k}_{(0, r)}, \forall r \leq   r_{t, t'}.$$

Applying the same arguments as in the proof of Proposition \ref{prop_main_Killing} (Case 1) (just replace $\varepsilon$ with $\|t - t'\|$ and consider $t'$ as the origin) we obtain the desired result.
\end{proof}
By using Proposition \ref{prop_main_killing_w} and  the same arguments as in the proof of Theorem \ref{thm_killing_b} we get:
\begin{thm} The local Lipschitz Killing curvatures of $A$ are locally Lipschitz along the strata of $\Sigma$.
\end{thm}

\section{Acknowledgement} We would like to thank Professor David Trotman for his interest and suggestion for the problem.

\bibliographystyle{babplain}

\begin{thebibliography}{10}

\bibitem{bb}
A.~Bernig and L.~Br{\"o}cker, \emph{Lipschitz-{K}illing invariants}, Math.
  Nachr. \textbf{245} (2002), 5--25. \MR{1936341 (2004a:53098)}

\bibitem{comte1}
G.~Comte, \emph{\'{E}quisingularit\'e r\'eelle: nombres de {L}elong et images
  polaires}, Ann. Sci. \'Ecole Norm. Sup. (4) \textbf{33} (2000), no.~6,
  757--788. \MR{1832990 (2002d:32040)}

\bibitem{cm}
G.~Comte and M.~Merle, \emph{\'{E}quisingularit\'e r\'eelle. {II}. {I}nvariants
  locaux et conditions de r\'egularit\'e}, Ann. Sci. \'Ec. Norm. Sup\'er. (4)
  \textbf{41} (2008), no.~2, 221--269. \MR{2468482 (2010c:32049)}

\bibitem{coste}
M.~Coste, \emph{An introduction to o-minimal geometry}, Dip. Mat. Univ. Pisa,
  Dottorato di Ricerca in Matematica, Istituti Editoriali e Poligrafici
  Internazionali, Pisa, 2000.

\bibitem{plessis}
A.~A. du~Plessis, \emph{Continuous controlled vector fields}, Singularity
  theory ({L}iverpool, 1996), London Math. Soc. Lecture Note Ser., vol. 263,
  Cambridge Univ. Press, Cambridge, 1999, pp.~xviii--xix, 189--197. \MR{1709353
  (2001i:58090)}

\bibitem{kp}
K.~Kurdyka and A.~Parusi{\'n}ski, \emph{Quasi-convex decomposition in o-minimal
  structures. {A}pplication to the gradient conjecture}, Singularity theory and
  its applications, Adv. Stud. Pure Math., vol.~43, Math. Soc. Japan, Tokyo,
  2006, pp.~137--177. \MR{2325137 (2008e:32009)}

\bibitem{kr}
K.~Kurdyka and G.~Raby, \emph{Densit\'e des ensembles sous-analytiques}, Ann.
  Inst. Fourier (Grenoble) \textbf{39} (1989), no.~3, 753--771. \MR{1030848
  (90k:32026)}

\bibitem{loi3}
T.~L. Loi, \emph{Lecture 1: o-minimal structures}, The {J}apanese-{A}ustralian
  {W}orkshop on {R}eal and {C}omplex {S}ingularities---{JARCS} {III}, Proc.
  Centre Math. Appl. Austral. Nat. Univ., vol.~43, Austral. Nat. Univ.,
  Canberra, 2010, pp.~19--30. \MR{2763233 (2012h:32011)}

\bibitem{mather}
J.~Mather, \emph{Notes on topological stability}, Bull. Amer. Math. Soc. (N.S.)
  \textbf{49} (2012), no.~4, 475--506. \MR{2958928}

\bibitem{nhan1}
N.~Nguyen, \emph{Structure m\'etrique et g\'eom\'etrie des ensembles
  d\'efinissables dans des structures o-minimales}, Ph.D. Thesis, In
  preparation, 2015.

\bibitem{ot}
P.~Orro and D.~Trotman, \emph{C\^one normal et r\'egularit\'es de
  {K}uo-{V}erdier}, Bull. Soc. Math. France \textbf{130} (2002), no.~1, 71--85.
  \MR{1906193 (2003f:58006)}

\bibitem{tv}
D.~Trotman and G.~Valette, \emph{On the local geometry of definably stratified
  sets}, In preparation, 2013.

\bibitem{valette1}
G.~Valette, \emph{Volume, {W}hitney conditions and {L}elong number}, Ann.
  Polon. Math. \textbf{93} (2008), no.~1, 1--16. \MR{2383338 (2009b:32011)}

\bibitem{dries}
L.~van~den Dries, \emph{Tame topology and o-minimal structures}, London
  Mathematical Society Lecture Note Series, vol. 248, Cambridge University
  Press, Cambridge, 1998. \MR{1633348 (99j:03001)}

\bibitem{dm1}
L.~van~den Dries and C.~Miller, \emph{Geometric categories and o-minimal
  structures}, Duke Math. J. \textbf{84} (1996), no.~2, 497--540. \MR{1404337
  (97i:32008)}

\bibitem{verdier}
J.~L. Verdier, \emph{Stratifications de {W}hitney et th\'eor\`eme de
  {B}ertini-{S}ard}, Invent. Math. \textbf{36} (1976), 295--312. \MR{0481096
  (58 \#1242)}

\end{thebibliography}
\providecommand{\bysame}{\leavevmode\hbox to3em{\hrulefill}\thinspace}
\providecommand{\MR}{\relax\ifhmode\unskip\space\fi MR }
\providecommand{\MRhref}[2]{%
  \href{http://www.ams.org/mathscinet-getitem?mr=#1}{#2}
}
\providecommand{\href}[2]{#2}

\end{document}